\documentclass[a4paper,12pt]{article}

\usepackage{a4,amssymb,amsmath,amsthm,url,xcolor,enumerate,float,lineno}
\usepackage{bezier,amsfonts,amssymb,graphicx,amsthm,url}
\usepackage[utf8]{inputenc}
\usepackage{amsmath}
\usepackage{amsfonts}
\usepackage{amssymb,amsthm}
\usepackage{url}

\newtheorem{theorem}{Theorem}[section]
\newtheorem{definition}{Definition}[section]
\newtheorem{proposition}[theorem]{Proposition}
\newtheorem{corollary}[theorem]{Corollary}
\newtheorem{lemma}[theorem]{Lemma}
\newtheorem*{theorem*}{Theorem}
\newtheorem*{problem*}{Problem}

\def\kaxxa{{\vcenter {\hrule height .2mm
\hbox{\vrule width .2mm height 2mm \kern 2mm
\vrule width .2mm} \hrule height .2mm}}}

\title{The feasibility problem  - the family ${\cal F}$$(G)$ of all induced $G$-free graphs.   
}
\author{\begin{tabular}{cc}Yair Caro & Matthew Cassar \\ University of Haifa-Oranim & University of Malta\\\\ Josef Lauri & Christina Zarb \\ University of Malta & University of Malta \end{tabular}}
\date{ }

\DeclareGraphicsExtensions{.pdf,.png,.jpg}

\begin{document}



\maketitle

\begin{abstract}

An infinite family of graphs ${\cal F}$ is called feasible if for any pair of integers $(n,m)$, $n \geq 1$, $0  \leq m \leq \binom{n}{2}$, there is a member $G \in  {\cal F}$ such that $G$ has $n$ vertices and $m$ edges.
 
We prove that given a graph $G$, the family ${\cal F}$$(G)$ of all induced $G$-free graphs is feasible if and only if $G$ is not $K_k$, $K_k\backslash K_2$, $\overline{K_k}$, $\overline{K_k\backslash K_2}$, for $k \geq 2$.
\end{abstract}
 
\section{Introduction}

The \emph{Feasibility Problem} is an umbrella for various  specific problems in extremal combinatorics:  Let ${\cal F}$ be an infinite family of graphs.  Then ${\cal F}$ is called feasible if for every $n \geq 1$, $0  \leq m \leq \binom{n}{2}$,  there is  as graph $G \in {\cal F}$ having exactly $n$ vertices and $m$ edges.

If ${\cal F}$ is not feasible, it is of interest to find the set of all feasible pairs \[FP({\cal F})= \{ ( n ,m)  :  \mbox{ there is  a graph $G \in {\cal F}$ having exactly $n$ vertices and $m$ edges}\},\]  as well as the complementary set  \[\overline{FP}({\cal F})= \{  (  n, m)  : \mbox{ no member of ${\cal F}$ has precisely $n$ vertices and $m$ edges}\}.\] If it is not possible to exactly determine these sets, we look for good estimates of $h(n,{\cal F}) = \frac{|FP({\cal F})|}{\binom{n}{2}}$. and $g(n,{\cal F}) =  \frac{|\overline{FP}({\cal F})|}{\binom{n}{2}}$.  

Also in many cases in extremal graph theory it is of interest to find \[f(n,{\cal F})  =  \min \{m: ( n ,m) \mbox{ is not a feasible pair for the familiy ${\cal F}$}\}\] as well as \[F(n,{\cal F}) = \max \{m: ( n ,m) \mbox{ is not a feasible pair for the family ${\cal F}$}\}.\]
 
A simple example is  ${\cal F}$, the family of all connected graphs.  Clearly  every connected graph on $n$ vertices must have at least $n-1$ edges  and it is trivial to see that with the above notation, $f(n,{\cal F}) = 0$ and $F(n,{\cal F}) =  n-2$.  Another example is the family ${\cal F}$ of all planar graphs. Here it is well  known that $f(n,{\cal F}) = 3n-5$ for $n \geq 4$ (since a maximal planar graph can have at most $3n -6$ edges for $n \geq 3$) and $F(n,{\cal F}) =  \binom{n}{2}$  for $n \geq 5$. In both of these examples, the exact determination of $FP({\cal F})$ and $\overline{FP}({\cal F})$ as well as $h(n,{\cal F})$ and $g(n,{\cal F})$ is easy.

A further important example is the celebrated  problem of Tur{\'a}n numbers $ex(n,G)$ \cite{turan1941external,turan1954theory}, which is the maximum number of edges  in a graph on $n$ vertices which does not have $G$ as a subgraph.


Clearly, with the notation above where ${\cal F}$ is the family of all $G$-free  graphs,  $ex(n,G)=  \min\{  f(n, {\cal F}) - 1, \binom{n}{2}\}$.  

Also for the class ${\cal F}$ of $G$-free graphs,  $g(n,{\cal F}) \rightarrow 1$ if $G$ is a bipartite graph, while  $g(n,{\cal F}) \rightarrow \frac{1}{2(\chi(G) - 1)}$ otherwise (by Erd{\"o}s-Stone-Simonvits  theorem \cite{erdos1966limit,stone1946structure}),  where $\chi(G)$ is the chromatic number of $G$. 
For references to extremal graph theory we  refer to  \cite{alon2016probabilistic,bollobas2004extremal,furedi2013history}.     

Erd{\"o}s, Furedi, Rothschild and Sos \cite{erdHos1999induced} initiated a study of classes of graphs that forbid every induced subgraph on a given number $m$ of vertices and number $f$ of edges.  They used the notation  $(n, e) \rightarrow (m, f)$ if every graph $G$ on $n$ vertices and  $e$ edges has an induced subgraph on $m$ vertices and $f$ edges,  and they looked for pairs for which this relation does not hold, calling them avoidable pairs.   
 
So if we define $Q = Q(m,f)  =  \{  G :  |G| = m , e(G) = f \}$,  then, in our notation, the family ${\cal F}$ considered above is the family ${\cal F}(Q)$ of all $G$-free  graphs where $G  \in  Q$.  We emphasize here that the main interest in this  line of research is to estimate  a density measure defined by  \[\sigma(m, f) = \lim_{n \rightarrow\infty} \frac{|\{e : (n, e) \rightarrow (m, f)\}|}{\binom{n}{2}}\] (along the lines indicated by the above definition of $g(n,{\cal F})$), and the proofs incorporate number theoretic arguments.  

It is known that if $(m, f) \in \{(2, 0),(2, 1),(4, 3),(5, 4),(5, 6)\}$, then $\sigma(m, f) = 1$; otherwise, $\sigma(m, f) \leq 1/2$  (see the references above). Also, Erd{\"o}s et al. gave a construction that shows that for most pairs $(m, f)$ we have $\sigma(m, f) = 0$.  For recent papers on this highly active subject we refer to \cite{axenovich2022unavoidable,axenovich2021absolutely,he2023improvements,weber2022unavoidable}. 

Yet another example is given in the paper \cite{caro2023feasibility} by the authors --- the feasibility problem for  line graphs --- where we solved completely $\overline{FP}({\cal F})$ and hence  $FP({\cal F}$)  when ${\cal F}$ is the family of all line graphs . In particular the values of $f(n,{\cal F})$ and $F(n,{\cal F})$ are exactly determined for the family ${\cal F}$ of all line graphs.

Reznick \cite{reznick1989sum} solved asymptotically, via a number theoretic approach, the value of $f(n,{\cal F}) = \frac{n^2}{2}  - \sqrt(2)n^{3/2} + O(n^{5/4)})$  where ${\cal F}$ is the family of all induced $P_3$-free graphs (corresponding to $(m,f) = (3,2)$  where clearly $\sigma(3, 2) = 0$), which are graphs represented as a vertex disjoint union of cliques, and his method is still in use in the research about the family $Q(m,f)$ defined above.    

The same order of magnitude is proved for $f(n,{\cal F})$  in case  where ${\cal F}$ is the family of all line graphs of acyclic graphs, as well as the family of all line graphs \cite{caro2023feasibility}.

Here, inspired in part by the problem launched by Erd{\"o}s et al. concerning ${\cal F}(Q)$, and as a counterpart to the Tur{\'a}n  problem concerning families with no subgraph isomorphic to $G$, we consider the case where ${\cal F}  = {\cal F}(G)$ is the family of all induced $G$-free graphs.  Clearly, if $G \in \{ K_k, K_k\backslash K_2, \overline{K_k}, \overline{K_k \backslash K_2}\}$ for $ k \geq 2$, then ${\cal F}(G)$ is trivially non-feasible, hence we use in the sequel TNF$=\{ K_k, K_k\backslash K_2, \overline{K_k}, \overline{K_k \backslash K_2}\}$ for $k \geq 2$.  We prove the following, our main theorem, using only graph theoretic arguments: 
\begin{theorem*}[\textbf{Main}]
Let $G$ be a graph --- the family ${\cal F}(G)$ of all induced $G$-free graphs is feasible if and only if $G$ is not a member of  \[\mbox{TNF} = \{ K_k, K_k\backslash K_2, \overline{K_k}, \overline{K_k \backslash K_2}\}\]  for $ k \geq 2$.
\end{theorem*}

In other words, if $G$ is not a member of TNF, then for every pair $(n,m)$, $n \geq 1$, $0  \leq m \leq \binom{n}{2}$, there is an induced $G$-free graph with exactly $n$ vertices and $m$ edges.

While $f(n,{\cal F})$, $F(n,{\cal F})$, $\overline{FP}({\cal F})$ and $FP({\cal F})$ are  determined by Tur{\'a}n 's Theorem for the cases ${\cal F} = {\cal F}(K_k)$ and  ${\cal F} = {\cal F}(\overline{K_k})$, determining $f(n,{\cal F})$ and $\overline{FP}({\cal F})$ for  ${\cal F} = {\cal F}(K_k\backslash K_2)$ and $F(n,{\cal F})$ and $ \overline{FP}({\cal F})$ for $ {\cal F} = {\cal F}(\overline{K_k \backslash K_2})$ are not yet solved.  

We have already started considering the general feasibility problem in our paper \cite{caro2023feasibility} where we proved that several natural families of graphs  are feasible, namely  $K_{1,r}$-free graphs for $r \geq 3$, $P_r$-free graphs for $r \geq 4$, $rK_2$-free graphs for $r \geq 2$,  as well the family of chordal graphs, and the family  of paw-free graphs.

In the rest of the paper, when we say $G$-free we mean induced $G$-free.

Our paper is organized as follows:  in section 2 we  discuss some basic properties of  feasible families  with regards  to containment and complementation.  We then introduce the two main constructions crucial for the proof of the main theorem. The first is the UEP (Universal Elimination Process), first introduced in \cite{caro2023feasibility}  and the second is the $\{K_3,K_2\}$-elimination process.  We shall discuss some consequences of these constructions.

In section 3 we prove the main theorem of this paper.

In section 4  we  offer  interesting examples and questions for further research.

\section{Feasible Families under containment and complementation and elimination procedures}

\subsection{ Basic properties}

The following are simple basic properties concerning feasibility subject to containment and complementation. The proofs are easy but we include them for the sake of completeness.

\begin{proposition} \label{prop1}
Let ${\cal F}$ and ${\cal H}$ be two families of graphs  such that  ${\cal H} \subset  {\cal F}$.  Then
\begin{enumerate}
\item{If ${\cal H}$ is a feasible family then ${\cal F}$ is feasible family.}
\item{If ${\cal F}$ is not a feasible family then ${\cal H}$ is not feasible family.}
\end{enumerate}

 \end{proposition}
The proof is trivial.

\begin{proposition} \label{prop2}
 $\mbox{ }$
 \begin{enumerate}
\item{Let ${\cal F}$ be a family of graphs and $ \overline{{\cal F}} =  \{  \overline{G} : G \in  {\cal F} \}$.  Then ${\cal F}$ is feasible if and only if $\overline{{\cal F}}$ is feasible.}
\item{Let ${\cal F}(G)$ be the family of all induced $G$-free graphs . Then $\overline{{\cal F}} = {\cal F}(\overline{G})$, the family of all induced $ \overline{G}$-free graphs, is feasible if and only if ${\cal F}(G)$ is feasible.}
\end{enumerate}
\end{proposition}

\begin{proof}
$\mbox{ }$
 \begin{enumerate}
\item{Suppose ${\cal F}$ is feasible:  given a pair $( n ,m)$  $n \geq 1$, $0 \leq m \leq \binom{n}{2}$, there is a graph  $G \in {\cal F}$  having $n$ vertices and  $m$ edges.  Clearly $\overline{G}$ has $n$ vertices and $\binom{n}{2} - m$ edges.  Then as $m$ increase  from $0$  to $\binom{n}{2}$, $\binom{n}{2}  - m$ decrease from $\binom{n}{2}$ to $0$.

Hence for every pair $(n,m)$  there exists $\overline{G} \in  \overline{{\cal F}}$   having $n$ vertices and $m$ edges.  The other direction is symmetric.}

\item{Suppose $H \in  {\cal F}(G)$ is induced $G$-free. Consider $\overline{H}$ ---  if it contains an induced copy of $\overline{G}$, then $H$ would contain an induced copy of $G$.  The other direction is symmetric.}
\end{enumerate}
\end{proof}

\begin{proposition} \label{propfeas3}
Let $G$ and $H$ be two graphs with $H$ an induced subgraph of $G$.  Let ${\cal F}$$(G)$ and ${\cal F}(H)$ be, respectively, the families of all induced $G$-free and induced $H$-free graphs .
 \begin{enumerate}
\item{If ${\cal F}(H)$  is feasible  then ${\cal F}(G)$ is feasible.}
\item{If ${\cal F}(G)$ is not feasible then ${\cal F}(H)$ is not feasible.}
\end{enumerate}
\end{proposition}
\begin{proof}
Observe that since $H$ is an induced subgraph of $G$, a graph $P$ which is induced $H$-free  is also induced $G$-free  because if $P$ contains an induced copy of $G$ then it must contain induced copy of $H$  in the induced copy of $G$.   Hence ${\cal F}(H)  \subset  {\cal F}(G)$  and we apply proposition \ref{prop1}.

\end{proof}

\subsection{  The Universal Elimination Process (UEP)  and its consequences}

The Universal Elimination Process (UEP), introduced in \cite{caro2023feasibility},  is a method which is used to delete edges systematically from a complete graph.  We describe UEP here again for the sake of being self-contained.  We start with $K_n$  and order the vertices $v_1,\ldots, v_n$.  We now delete at each step an edge incident with $v_1$ until $v_1$ is isolated.  We then repeat the process of step by step deletion of the edges incident with $v_2$, and continue until we reach  the empty graph on $n$ vertices. 

 Along the process, for any pair $( n,m)$, $ 0  \leq m \leq \binom{n}{2}$, we have  a graph $G$ with $n$ vertices and $m$ edges.
\begin{lemma}[\cite{caro2023feasibility}]

The maximal induced subgraphs of $K_n$ obtained when applying  UEP on $K_n$  are of the form {\color{black} $H(p,q,r) =  ( K_p  \backslash K_{1,q} )  \cup rK_1$},  $p-1 \geq q \geq 0$  and $p +r = n$.   
\end{lemma}
 
\begin{proof}
This is immediate from the definition and description of UEP.
\end{proof}

Figure \ref{fig1} shows examples of $H(p,q,r)$ graphs.

\begin{figure}[H]
\centering
\includegraphics[scale=1]{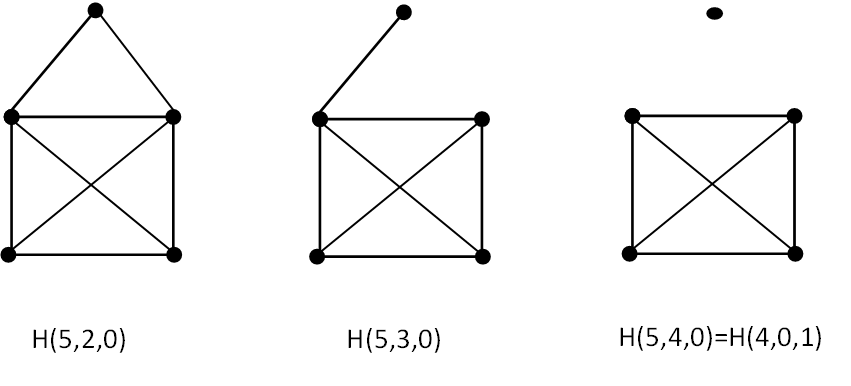} 
\caption{Examples of $H(p,q,r)$ graphs}
\label{fig1}
\end{figure}

Already the  UEP supplies many feasible families as summarized in the following corollary.

\begin{corollary} [\cite{caro2023feasibility}]
The following families of graphs obtained by  applying the  UEP  are feasible:
\begin{enumerate}
\item{ induced  $K_{1,r}$-free for $r\geq 3$, where $K_{1,r}$ is the star with $r$ leaves.}
\item{  induced  $P_r$-free  for $r \geq 3$, where $P_r$ is the path on $r$ edges.}
\item{ induced  $rK_2$-free for $r \geq 2$ where $rK_2$ is the union of $r$ disjoint edges.} 
\end{enumerate}   
\end{corollary}
\begin{proof}
This is immediate from the definition and description of UEP.
\end{proof}

\begin{definition}
For non-negative integers $p$ and $r$, $p+r  \geq 2$, let $S(p,r)$ denote the complete split graph   $K_p + \overline{K_r}$,  namely a clique $K_p$ and an independent set  $\overline{K_r}$ and all edges between the vertices in $K_p$ and $\overline{K_r}$. 
\end{definition} 

Observe that $S(p,r) = \overline{H(r ,0 ,p)}$.  We give some results related to the feasibility of $F(S(p,r))$.

 \begin{lemma} \label{feasSpr}
The feasibility of ${\cal F}(S(p,r))$:
\begin{enumerate}
\item{For $p = 0$  or $r = 0 $,  ${\cal F}(S(p,r))$ is not feasible.}
\item{For $p \geq 1$, $r \in \{ 1 ,2\}$,  ${\cal F}(S(p,r))$ is not feasible.}  
\item{For $p \geq 1$, $r \geq 3$,  ${\cal F}(S(p,r))$ is feasible.}
\end{enumerate}
\end{lemma}

\begin{proof}
$\mbox{ }$\\
\begin{enumerate}
\item{ This is because $S(p,r)$  is  either  $K_p$ or $\overline{K_r}$ which and clearly ${\cal F}(G)$ is not feasible when $G=K_k$ or $G=\overline{K_k}$.}
\item{This is because $S(p,1) = K_{p+1}$ and $S(p,2) =  K_{p+1} \backslash K_2$, and again ${\cal F}(G)$ is  clearly not feasible when $G=K_k \backslash K_2$ .}
\item{This is because $S(1,r)  = K_{1,r}$ and we already proved in \cite{caro2023feasibility} (and mentioned before)  that  ${\cal F}(K_{1,r})$ is feasible for $r \geq 3$.  Also since the UEP produces  $K_{1,r}$-free graphs for $r \geq 3$,  it follows  that for $p \geq 2$, $r \geq 3$, $S(p,r)$  is not an induced subgraph in any graph obtained by the UEP.   }
\end{enumerate}
\end{proof}  

The following is an immediate application of Proposition \ref{propfeas3}  and the fact that $S(p,r) =  \overline{H(r, 0, p )}$, as well as  Lemma \ref{feasSpr} by replacing the role of $r$ and $p$ due to complementation.
 
\begin{corollary} \label{cory}

The feasibility of ${\cal F}(H(p,0,r))$.
\begin{enumerate}
\item{For $p = 0$  or $r = 0$,  ${\cal F}(H(p,0,r))$ is not feasible.} 
\item{For $p \in \{ 1 ,2 \}$  and $r \geq 1$,   ${\cal F}(H(p,0,r))$ is not feasible.}  
\item{For $p \geq 3$, $ r \geq 1$, ${\cal F}(H(p,0,r))$ is feasible.}
\end{enumerate}
\end{corollary}
\begin{proof}
$\mbox{ }$\\
\begin{enumerate}
\item{This is because $H(p,0,r)$ in this case is a clique or an independent set.}
\item{This is because $H(p,0,r)$ in this case is the independent set $\overline{K_{r+1}}$  for $p  = 1$, and  $K_2 \cup \overline{K_r}$ for $p=2$,  both members of TNF.}
\item{This follows by complementation.}
\end{enumerate}
\end{proof}

This elimination method, however, does not work in the case of a family ${\cal F}(G)$ of induced $G$-free graphs when $G$ is of the form $H(p,q,r)$.  In \cite{caro2023feasibility}, the authors prove that the family of paw-free graphs is feasible, where the paw graph is isomorphic to $H(4,2,0)$.  They use a different edge elimination technique, which we develop and extend in the next section.

\subsection{$  \{K_3,K_2\}$-elimination  and its consequences.}
\begin{lemma} \label{K3K2}
\cal{{\cal F}}or $n\geq 2$ and $0 \leq t \leq  n-2$, there are  integers $x ,y \geq 0$ such that $3x +y = t$  and $xK_3 \cup yK_2$ is a subgraph of $K_n$.   
\end{lemma}
\begin{proof}

Clearly this is true by direct checking for $n = 2$  with $(x,y) = (0 ,0)$, $n =3$  with $(x,y) \in \{(0,0),(0,1)\}$, $n=4$ with $(x,y) \in \{(0,0),(0,1),(0,2)\}$ and $n = 5$ with $(x,y) \in \{(0,0),(0,1),(0,2)(1,0)\}$.

 So assume $n \geq 6$ and write $n = 3k +r$, $0 \leq r \leq 2$ and $k \geq 2$.  We consider three cases:
\begin{enumerate}
\item{When $r =  0$, $n = 3k$.  For $0\leq t \leq  n-3$  we shall consider $K_{n-1}$ in $K_n$ and by  induction any $t$  in this range can be represented by $xK_3 \cup yK_2$ as a subgraph of $K_{n-1}$ hence of $K_n$. 

So we only need to consider $t = n-2$.   We take $(k-1)K_3$  ($k \geq 2$) that covers $3k-3$ vertices hence from the remaining three vertices forming $K_3$ we can choose $K_2$ and we get $(k-1)K_3 \cup K_2$ on $3k-2  = n-2$ edges.  }
\item{When $r  = 1$, $n = 3k +1$.   As above, for $0\leq t \leq  n-3$ we shall apply  induction on $ n-1$ vertices.  So we need consider only  $t= n -2  = 3k-1$.  We take $( k-1)K_3$ that cover $3k - 3$ vertices and from the remaining 4 vertices forming $K_4$ we choose $2K_2$ and get $( k-1)K_3 \cup 2K_2$ on $3k  - 1 = n-2$ edges.  }  
\item{When $r  = 2$, $n = 3k +2$.   As above, for $0\leq t \leq  n-3$ we shall apply  induction on $ n-1$ vertices.  So we need consider only  $t= n -2  = 3k$.  We take $kK_3$ that cover $3k$ vertices and get $3k = n-2$ edges.  }  
\end{enumerate}
\end{proof}

The $\{K_3,K_2\}$-elimination process  is described as follows:  starting from $K_n$,  for every  $0 \leq t \leq n-2$, delete edges in the form $xK_3 \cup yK_2$ such that $3x +y = t$.  Once this is done , we have covered all the range $[ \binom{n-1}{2}  +1, \ldots, \binom{n}{2} ]$.  Consider now $K_{n-1} \cup K_1$ (obtained by deleting a star $K_{1,n-1}$) and apply the $\{K_3,K_2\}$-elimination process on $K_{n-1}$ and continue until all edges are deleted.  Once again observe that this process covers all possible numbers of edges in the range $[0, \binom{n}{2}]$.

Observe that the graphs obtained through this elimination process are of the form $Q(p,r,x,y )  =   (K_p \backslash \{ xK_3 \cup yK_2 \}) \cup \overline{K_r} $ for $p,r, x ,y \geq 0$ and $0 \leq3x + y \leq p$ and  $p +r  = n$.

\begin{lemma}
The graphs obtained through the $(K_3,K_2)$-elimination process are  $H(p,q,r)$-free  for  $p \geq 4$, $r \geq 0$,  $2 \leq q \leq p-2$.  In particular the family ${\cal F}$ of all $H(p,q,r)$-free graphs is feasible whenever $p\geq 4$, $r \geq 0$,  $2 \leq q \leq p-2$.
\end{lemma}
 
\begin{proof}

The proof is by comparing the structure of $Q(p,r,x,y)$ versus $H(p,q,r)$ graphs .

The only cases where  $Q(p,r,x,y) = H(p,q,r)$, ($3x +y  = q$) are  $q= 0  ,1$,  where the connected part is either $K_p$  or $K_p \backslash K_2$, since the $q$ edges in $H(p,q,r)$ are deleted via deletion of a star on $q$ edges.   In all other cases  $Q(p,r,x,y)$ graphs are $H(p,q,r)$-free graphs  and since $Q(p,r,x,y)$ with  $p+r  = n$  covers  all possible values of $m$ in the range $0  \leq m \leq \binom{n}{2}$ via the $\{K_3,K_2\}$-elimination process, it follows that  for fixed  $p \geq 4$, $r$, $2 \leq q \leq  p-2$, $r \geq 0$, the family ${\cal F}$  of all  $H(p,q,r)$-free  graphs is feasible.

 \end{proof}

\section{Concluding the proof of the main Theorem}

We shall now complete the proof of the main Theorem.  Observe that by the UEP and $\{K_3,K_2\}$-elimination process together with the determination of the feasibility of ${\cal F}(H(p,0,r))$ in section 2, what remains to consider is the feasibility of ${\cal F}(G)$ where $G$ is an $H(p,q,r)$ graph with $p \in \{2,3\}$  and in the case when $p \geq 4$  and  $q = 1$  or $q = p-1$  with $r \geq 0$ (and their complements which follow by Proposition \ref{prop2}).

\begin{proposition}\label{prop31}
 The case  $p = 2$.
\end{proposition}
 \begin{proof}

Observe that $p=2$  gives either $H(2,0,r)  =  K_2 \cup  \overline{K_r}$, or $H(2,1,r)  =  \overline{K_{r +1}}$  which belong to the family TNF.
\end{proof}

\begin{proposition} 
The case $p = 3$.
\end{proposition}

 \begin{proof}

Observe that $p = 3$ gives   $H(3,0 , r)$, $  H(3,1,r)$,  $H(3,2,r) = H(2,0,r+1)$. We consider each of these graphs:
\begin{enumerate}
\item{If $G = H(3,0,r)$ then if $r = 0$, $G = K_3$ belongs to TNF and ${\cal F}(G)$ is not feasible, while if $r \geq 1$  then by Corollary \ref{cory} part 2,  ${\cal F}(G)$ is feasible.  }
\item{If $G=  H(2,0,r+1)$, ${\cal F}(G)$ is not feasible by Proposition \ref{prop31}.}
\item{If $G =  H(3,1,r)$ then if $r = 0$,  $G =  K_3\backslash K_2$ which is  a member of TNF and hence not feasible.  If $r \geq 1$ then  $G = K_3\backslash K_2 \cup \overline{K_r}$.  When $r = 1$, $G$ is the complement of the paw-graph, i.e. $G=\overline{H(4,2,0)}$ which is feasible by Proposition \ref{prop2} and hence ${\cal F}(G)$ is feasible.  For $r  \geq 2$, $ G = \overline{K_{p+r} \backslash  K_{1,2}}$ which is feasible by $\{K_3,K_2 \}$-elimination since $p +r \geq 5$  and we can delete $2K_2$.}
\end{enumerate}
\end{proof}

\begin{proposition}
The case $p \geq 4$.
\begin{enumerate}
\item{For $ p \geq 4$ and $q = p-1$,  $H(p,p-1,r)$  is feasible for $r \geq 0$.}
\item{For $ p \geq 4$ and $q = 1$, $H(p,1,r)$  is feasible for $r \geq 1$ and not feasible for $r = 0$  (a member of TNF)}
\end{enumerate}
\end{proposition}

 \begin{proof}
 
$\mbox{ }$\\
\begin{enumerate}
\item{If $p \geq 4$ and $q = p-1$ then $G= H(p,p-1,r) = H(p-1,0,r+1) = K_{p-1} \cup \overline{K_{r +1}}$  and we are done by Corollary \ref{cory}.}  
\item{If $p \geq 4$, then if $q = 1$, $G =  H(p ,1, r) =   K_p \backslash  K_2  \cup \overline{K_r}$.  If $r = 0$,  $H(p,1,0) = K_p\backslash K_2$ is a member of TNF hence ${\cal F}(G)$ is not feasible.  So we may assume that $r \geq 1$.  Recall that the family of claw-free graph is feasible by UEP.  For $p \geq 4$,  the complement  of  $H(p,1,r)$ with $r \geq 1$, $H$,  contains an induced claw.  So a  claw-free graph cannot have  $H$  as an induced graph, hence it is in particular $H$-free.  Since the family of all claw-free graphs is feasible and $H$-free,  it follows that the family of all $H$-free graphs (containing the family of claw-free graphs)  is feasible.  Therefore, applying Proposition  \ref{propfeas3} we get that the family of all  $H(p,1,r)$-free graphs with $r \geq 1$ is feasible. }
\end{enumerate}
\end{proof} 

Hence we have proved that ${\cal F}(G)$ is feasible if and only if  $G$ is not one of the graphs $ K_k$, $K_k\backslash K_2$, $\overline{K_k}$, $\overline{K_k\backslash K_2}$.

\section{Further Examples and Open Problems}

After the proof of the main Theorem, a natural question is the following: Suppose $G$ and $H$ are graphs such that ${\cal F}(G)$ and ${\cal F}(H)$ are both feasible families.  Is ${\cal F}(G,H)$, the family of all  graphs which are simultaneously  induced $G$-free and induced $H$-free, necessarily feasible ?

We know if both $G$ and $H$ are not $H(p,q,r)$ graphs then ${\cal F}(G,H)$ is feasible by UEP.  Since $H(p,q,r)$ graphs on three vertices belong to the  TNF family,  then the smallest interesting case is $H(4,2,0)$, the  Paw graph.

The following answers the above question negatively despite the fact that ${\cal F}(Paw)$ and ${\cal F}(K_{1,3})$ are both feasible families as proved in \cite{caro2023feasibility}. 
\begin{theorem}

${\cal F}(Paw , Claw )= {\cal F}(H(4,2,0),K_{1,3})$  is not a feasible family.  Also  ${\cal F}(P_3 \cup K_1 , K_3 \cup K_1)$ is not feasible.
\end{theorem}
\begin{proof}
These families are complementary and hence by Proposition \ref{prop2} the two statements are equivalent.  It is rather easy to check that the pair $( n,m)  = ( 5,3)$  forces an induced $P_3 \cup K_1$ or $K_3 \cup K_1$ and hence the pair $( n,m) =  (5,7)$  forces an induced  $Paw$  or $Claw$.   It is also still easy to check that the pair $(n,m) =  ( 6,4)$  forces an induced member of  $P_3 \cup K_1$ or $K_3 \cup K_1$and hence the pair $(n,m) =  ( 6,11)$ forces an induced $Paw$  or $Claw$. 

So clearly ${\cal F}(Paw , Claw )$ and  ${\cal F}(P_3 \cup K_1 , K_3 \cup K_1)$ are not feasible.

A more general argument is the following:    Suppose we consider a graph $G$ on $n \geq 5$ vertices and $m$  edges, $\lfloor \frac{n}{2} \rfloor +1 \leq m \leq n - 2$.   Then $G$ contains  $P_3$ as we cannot pack $mK_2$.  Clearly no  such graph on $m$  edges is connected for $n \geq 5$.
 
Consider the connected component  $B$ containing this $P_3$,  and a vertex $v \in V \backslash B$.  If $B$ is not a complete graph is must contain induced $P_3$ together with $v$ forming the induced subgraph $P_3 \cup  K_1$.  If $B$ is a clique it must be of order at least 3 as $B$ contains $P_3$.  But in this  case, together with $v$ we have $K_3 \cup  K_1$ as an induced  subgraph.   Hence  for  $n \geq 5$,  the pair $( n ,m)$  where  $\lfloor \frac{n}{2} \rfloor +1 \leq m \leq n - 2$ is not a feasible pair for  ${\cal F}(P_3 \cup K_1 , K_3 \cup K_1)$.

Hence by considering the complement, for $n \geq 5$  the pair  $( n ,  \binom{n}{2}- m )$  where $\lfloor \frac{n}{2} \rfloor +1 \leq m \leq n - 2$,  is a non-feasible pair for ${\cal F}(Paw,  Claw)$.  

However the pairs $( n,m)$ where $0 \leq m \leq \lfloor \frac{n}{2} \rfloor$ are feasible for  ${\cal  F} =  {\cal F}(P_3 \cup K_1 , K_3 \cup K_1)$  since the graph $mK_2$ is both  $P_3 \cup K_1$ and $K_3 \cup K_1$ induced free.  Also observe that the pair $( n ,n -1)$  is feasible  for  ${\cal  F} =  {\cal F}(P_3 \cup K_1 , K_3 \cup K_1)$  by taking the graph $K_{1,n-1}$.
\end{proof}

Observe that graphs which are a union of cliques belong to ${\cal F} =  F(Paw , Claw )$  which forces that \[f(n,{\cal F}) \geq  \frac{n^2}{2} - \sqrt{2}n^{3/2} + O(n^{5/4})\]  as mentioned in the introduction,  and \[F(n, {\cal F})  =  \binom{n}{2}  - \left \lfloor \frac{n}{2}\right \rfloor -1 \mbox{ for }  n \geq 5.\]   

Hence by considering the complement, for ${\cal  F} =  {\cal F}(P_3 \cup K_1 , K_3 \cup K_1)$  we get
\[F(n,{\cal F}) \leq  \sqrt{2}n^{3/2} + O(n^{5/4})\]
 \[f(n, {\cal F})  =   \left \lfloor \frac{n}{2}\right \rfloor +1 \mbox{ for }  n \geq 5\]   

as proved above.   

\bigskip

\noindent \textbf{Problem}:  It would be interesting to improve the lower bound on $f(n, {\cal F})$  for ${\cal F} =  F(Paw , Claw )$  and the corresponding value of $F(n, {\cal F})$ for ${\cal  F} =  {\cal F}(P_3 \cup K_1 , K_3 \cup K_1)$. In particular is $F(n, {\cal F})$ linear in $n$ for ${\cal  F} =  {\cal F}(P_3 \cup K_1 , K_3 \cup K_1)$?

\bigskip

Another interesting question is:  since ${\cal F}(Paw , Claw )$ is not a feasible family, is ${\cal F}(Paw,K_{1,4})$ a feasible family or not, considering the fact that ${\cal F}(Paw , K_{1,3})  \subset {\cal F}(Paw,K_{1,4}) \subset  {\cal F}(Paw,K_{1,r})$  for $r \geq 5$.  We prove the following theorem to answer this question.

\begin{theorem}
 ${\cal F}(Paw,K_{1,4})$  is a feasible family and so is  ${\cal F}(Paw,K_{1,r})$ for $r \geq 5$.
\end{theorem}

\begin{proof}

By Proposition \ref{prop1}, if ${\cal F}(Paw,K_{1,4})$  is a feasible family then so is  ${\cal F}(Paw,K_{1,r})$ for $r \geq 5$.

We shall work with the complementary family ${\cal  F} =  {\cal F}(P_3 \cup K_1 , K_4 \cup K_1)$ and show that it is feasible.  Then  by Proposition \ref{prop2} ${\cal F}(Paw,K_{1,4})$  is also feasible.

We shall use the split graphs $K_p + \overline{K_{n-p}}$ for $0 \leq p \leq n-3$ to construct graphs in ${\cal F}(P_3 \cup K_1 , K_4 \cup K_1)$ which cover the range $0 \leq m \leq \binom{n}{2}-2$.  When $p = 0$, the graph is $\overline{K_n}$ with 0 edges.  We can pack a graph with $k$ edges of the form $aK_3 \cup bK_2 \cup cK_1$  with $ n =3a +2b +c $ and $k=3a +b $ for $0 \leq k \leq n-2$. This has been shown in the $\{K_3,K_2\}$-elimination process in Lemma \ref{K3K2}.  For p=1, the split graph is $K_1 + \overline{K_{n-1}}=K_{1,n-1}$ which has exactly $n-1$ edges.  Again, we  can pack, in the independent part $\overline{K_{n-1}}$ of order $n-1$, graphs with $k$ edges of the form $aK_3 \cup bK_2 \cup cK_1$  with $ n-1 =3a +2b +c $ and $k=3a +b $  for  $ 0 \leq k \leq n-3$. With these graphs we cover the range $n  - 1, \ldots, n - 1 + n -3 = 2n -4$. All these graphs are in ${\cal F}(P_3 \cup K_1 , K_4 \cup K_1)$.

In general, the split graph $K_p + \overline{K_{n-p}}$ has $\frac{p(n-1)+(n-p)p}{2}$ edges and we can pack, in the independent part $\overline{K_{n-p}}$ of order $n-p$, graphs with $k$ edges of the form with $aK_3 \cup bK_2 \cup cK_1$  with $ n - p =3a +2b +c $ and $k=3a +b $ for $0 \leq k \leq n-p-2$, and cover the range of values of $m$ from $\frac{p(n-1)+(n-p)p}{2}$ to $\frac{p(n-1)+(n-p)p}{2} + n-p-2$.  Again, all these graphs are in ${\cal F}(P_3 \cup K_1 , K_4 \cup K_1)$.

For the last value $p=n-3$, we cover the range \[\frac{(n-3)(n-1)+ 3(n-3}{2} = \frac{n^2-n-6}{2} = \binom{n}{2} -3\] up to \[\binom{n}{2} -3 + n-(n-3)-2 = \binom{n}{2} -3+1 =\binom{n}{2} -2\] as discussed. 

The final two values of $m$, which are $\binom{n}{2} -1$ and $\binom{n}{2}$ are covered by the graph $K_n\backslash K_2$ (which is in fact $K_{n-2}+\overline K_2$), and $K_n$ itself, both graphs being in ${\cal F}(P_3 \cup K_1 , K_4 \cup K_1)$.  Thus the whole range of edges $0 \leq m \leq \binom{n}{2}$ is covered and the family ${\cal F}(P_3 \cup K_1 , K_4 \cup K_1)$ is feasible, as well as the family ${\cal F}(Paw,K_{1,4})$, and ${\cal F}(Paw,K_{1,r})$ for $r \geq 5$ by Proposition \ref{prop1}.

 \end{proof}
Another  problem concerns the family  ${\cal F}=  { \cal F}(K_4 \backslash K_2 )$, the smallest case in TNF for which the order of $f(n, {\cal F})$ is  interesting.  Clearly graphs which are union of cliques belong to ${ \cal F}(K_4 \backslash K_2 )$  hence \[f(n, {\cal F}) \geq  \frac{n^2}{2} - \sqrt{2}n^{3/2} + O(n^{5/4}) \]     
which also holds  for $f(n, {\cal F}) $ when ${\cal F}= {\cal F}(K_k\backslash K_2)$ and  $k \geq 3$. This is asymptotically sharp for $k = 3$ as we have  already seen. 

Note that this result is equivalent to the result in \cite{erdHos1999induced}, showing that $( n ,e) \rightarrow ( 4,5)$ is an avoidable pair.

The following  arguments give some more information on the non-feasible pairs $(n, m)$ for ${\cal F}(K_k \backslash K_2)$.  Clearly  a trivial upper bound  for $n \geq k$ is  $m = \binom{n}{2} - 1$  since the graph  $K_n \backslash K_2$ contains an induced $K_k\backslash K_2$   implying that the pair $( n ,m)$ is not feasible.  Hence $F(n,{\cal F}) = \binom{n}{2} - 1$ for ${\cal F}= {\cal F}(K_k\backslash K_2)$, $k \geq 3$.

Now suppose $G$ is a graph on $n$ vertices and $\binom{n}{2}-  t$ edges  where $ t \geq 1$ and   $n  -2t \leq  k-2$.  The missing $t$ edges can cover (in the complement) at most $2t$ vertices hence in $G$ there are at least $k-2$ vertices forming a clique  and adjacent to all vertices of $G$.  Choose a missing edge $e =xy$  then the $k - 2$ vertices  and $\{x , y \}$ form an induced $K_k \backslash K_2$.  So with $t =  \left \lfloor \frac{n- k+2}{2} \right \rfloor$  and $\binom{n}{2}- \left \lfloor \frac{n- k+2}{2} \right \rfloor  \leq  m \leq \binom{n}{2}- 1$ all the pairs $( n,m)$  are non-feasible for the family ${\cal F}(K_k \backslash K_2)$, proving that $f(n, {\cal F}) \leq \binom{n}{2}- \left \lfloor \frac{n- k+2}{2} \right \rfloor$  for this family,  and $F(n,{\cal F})\geq   \left \lfloor \frac{n- k+2}{2}\right \rfloor$ for the complementary family ${\cal F} = {\cal F}( K_2 \cup (k-2)K_1)$.
\bigskip

\noindent \textbf{Problem}: It would be interesting to improve upon the  lower bound for $f(n,{\cal F})$  for ${\cal F} = {\cal F}(K_k\backslash K_2)$,  as well as the corresponding upper-bound $F(n,{\cal F})$ for  ${\cal F} = {\cal F}( K_2 \cup (k-2)K_1)$.

\bibliographystyle{plain}
\bibliography{inducedfree}
\end{document}